\documentclass[12pt,reqno]{amsart}
\usepackage{amsthm, amsmath, amssymb, amscd, latexsym, multicol, verbatim, enumerate}
\usepackage[colorlinks=true, pdfstartview=FitV, linkcolor=blue, citecolor=blue, urlcolor=blue, breaklinks=true]{hyperref}

\usepackage[top=1.15in, bottom=1.15in, left=1.17in, right=1.17in]{geometry}
\usepackage[english]{babel}
\usepackage{tikz}
\usepackage[all]{xy}

\makeatother \theoremstyle{remark}
\numberwithin{equation}{section}

\theoremstyle{definition} 
\theoremstyle{definition}
\newtheorem{proposition}{Proposition}
\newtheorem{theorem}{Theorem}
\newtheorem{lemma}{Lemma}
\newtheorem{corollary}{Corollary}
\newtheorem{remark}{Remark}

\newtheorem{conjecture}{Conjecture}

\newcommand{\g}{\mathfrak{g}}
\newcommand{\n}{\mathfrak{n}}

\newcommand{\ra}{\rightarrow}

\newcommand{\vp}{\varphi}
\newcommand{\ts}{\otimes}

\newcommand{\mk}{\mathbb{K}}

\newcommand{\Hom}{\operatorname*{Hom}}

\newcommand{\lie}{\mathfrak}

\newcommand{\rep}{\operatorname*{rep}}

\newcommand{\Ext}{{\operatorname*{Ext}}}
\newcommand{\Rep}{\operatorname*{Rep}}
\newcommand{\bino}[2]{\left[\genfrac{}{}{0pt}{}{#1}{#2}\right]}

\begin{document}


\date{\today}



\title{Cones from quantum groups to tropical flag varieties}

\author{Xin Fang, Ghislain Fourier, Markus Reineke}
\address{Xin Fang:\newline
University of Cologne, Mathematical Institute, Weyertal 86--90, 50931 Cologne, Germany}
\email{xinfang.math@gmail.com}
\address{Ghislain Fourier:\newline
RWTH Aachen University, Pontdriesch 10--16, 52062 Aachen}
\email{fourier@mathb.rwth-aachen.de}
\address{Markus Reineke:\newline
Ruhr-Universit\"at Bochum, Faculty of Mathematics, Universit\"atsstra{\ss}e 150, 44780 Bochum, Germany}
\email{markus.reineke@rub.de}

\begin{abstract} We relate quantum degree cones, parametrizing PBW degenerations of quantized enveloping algebras, to (negative tight monomial) cones introduced by Lusztig in the study of monomials in canonical bases, to $K$-theoretic cones for quiver representations, and to some maximal prime cones in tropical flag varieties.
\end{abstract}

\maketitle

\section{Introduction}

The aim of this work is to point out a rather unexpected relation between two substantially different degeneration processes arising in the algebraic Lie theory.

The first is the study of toric degenerations of flag varieties $G/B$ of complex semisimple Lie algebras, the totality of which is encoded in tropical flag varieties with respect to Pl\"ucker embeddings. The second is the PBW-type degenerations of quantized enveloping algebras $U_q(\mathfrak{g})$ associated to complex semisimple Lie algebras. 

Our link between these two processes is provided by a class of polyhedral cones $\mathcal{D}_{\underline{w}_0}^q(\g)$, called {\it quantum degree cones} \cite{BFF,FFR16}, depending on the complex simple Lie algebra $\mathfrak{g}$ and a choice of $\underline{w}_0$, a reduced decomposition of the longest element in the Weyl group of $\mathfrak{g}$. For such a fixed reduced decomposition, the negative part $U_q(\mathfrak{n}^-)$ of the quantized enveloping algebra is generated by the quantum PBW root vectors with respect to some non-commutative straightening relations. The quantum degree cone is a kind of Gr\"obner fan in this non-commutative setup, where the monomial ordering is encoded in $\underline{w}_0$. Interior points of such cones give non-commutative degenerations of $U_q(\mathfrak{n}^-)$ to skew-polynomial algebras. As in the commutative Gr\"obner theory, due to lacking of knowledge on the straightening relations, the facets of the quantum degree cones are usually hard to describe.

The first main result of the present work is an embedding (conjectured to be an equality) of the quantum degree cone into the {\it negative tight monomial cone} of the Langlands dual Lie algebra, which is a variant of a cone introduced by Lusztig \cite{Lus2} for the identification of monomial elements in the canonical basis of $U_q(\mathfrak{n}^-)$.  This result is based on a detailed analysis of a result of Levendorskii and Soibelman \cite{LS1} on commutators of quantum PBW root vectors, for which we identify a certain PBW basis element appearing in this commutator with non-vanishing coefficient.

The second main result is a proof of the conjecture for reduced decompositions adapted to a Dynkin quiver. For the proof, we interpret the cone $\mathcal{D}_{\underline{w}_0}^q(\g)$ in terms of the $K$-theory of quiver representations, and use known results on degenerations and extensions of representations of Dynkin quivers \cite{B}.

Tropical flag varieties \cite{BLMM} are subfans of Gr\"obner fans of flag varieties with respect to Pl\"ucker embeddings. We observe that, via a result in \cite{FFFM}, the quantum degree cone arising from type $\tt A$ and a specific choice of reduced decomposition can be identified with a maximal prime cone in the corresponding tropical flag variety. This observation is quite surprising, as both cones arise from apparently different Gr\"obner theory: the straightening relations in the quantum setup become trivial after specialisation!

The authors hope that the present observation can be further developed into a more general link allowing the study of tropical flag varieties using the piece-wise linear combinatorics arising in quantized enveloping algebras.

The paper is organized as follows: in Section \ref{s2}, we first recall the construction of PBW bases in quantized enveloping algebras and state the commutation formula for root elements. Most of this section is devoted to the proof of the non-vanishing result (Theorem \ref{Thm:term}) stated above. This result allows us to relate the quantum degree cone and the negative tight monomial cone in Section \ref{s3}, and to formulate our main conjecture (Conjecture \ref{conjmain}). In Section \ref{s4} we review basic results on the representation theory of Dynkin quivers, and on the relation of their associated Hall algebras to the quantized enveloping algebra $U_q(\mathfrak{n}^-)$.
This allows us to give a $K$-theoretic interpretation of the quantum degree cone, and to prove our conjecture (Theorem \ref{Thm:LSQuiver}) in many cases. Finally, in Section \ref{s5} we rephrase the main result of \cite{FFFM} in the present language, identifying a maximal prime cone in the tropical flag variety with a particular quantum degree cone.

\vskip 10pt
\textbf{Acknowledgement.} X.F. would like to thank Gleb Koshevoy on discussions related to the tight monomial cones.

\section{Commutation relations of quantum PBW root vectors}\label{s2}

\subsection{Basics on Lie algebras and quantum groups}
Let $\g$ be a simple Lie algebra of rank $n$, with Cartan matrix $A=(a_{ij})_{n\times n}$. Let $\alpha_1,\ldots,\alpha_n$ be the simple roots and $\Delta_+$ be the set of positive roots. The root lattice will be denoted by $\mathcal{Q}$ and the weight lattice by $\mathcal{P}$. Let $(-,-)$ be the $W$-invariant scalar product on the weight lattice such that: for any short root $\alpha$, $(\alpha,\alpha)=2$; when $\g=\tt B_n, \tt C_n, \tt F_4$, for any long root $\beta$, $(\beta,\beta)=4$; for $\g=\tt G_2$ and the long root $\gamma$, $(\gamma,\gamma)=6$.
\par
For a positive root $\alpha\in\Delta_+$, we denote $d_\alpha:=(\alpha,\alpha)/2$, and for any simple root $\alpha_i$, $d_i:=d_{\alpha_i}$. Let $\alpha^\vee=d_\alpha^{-1}\alpha=2\alpha/(\alpha,\alpha)$. For $\lambda\in\mathcal{P}$ and $\alpha\in \mathcal{Q}$, $\langle\lambda,\alpha^\vee\rangle:=\frac{2(\lambda,\alpha)}{(\alpha,\alpha)}$. Then the entries of the Cartan matrix are given by: $a_{ij}=\langle\alpha_j,\alpha_i^\vee\rangle$.
\par
The simple reflections are defined by: $s_i(\mu)=\mu-\langle\mu,\alpha_i^\vee\rangle\alpha_i$. In particular, $s_i(\alpha_j)=\alpha_j-a_{ij}\alpha_i$.
\par
We fix several Lie theoretical notations. Let $W$ be the Weyl group of $\g$ generated by simple reflections $s_1,\cdots,s_n$ and the longest element $w_0\in W$. Let $R(w_0)$ be the set of all reduced decompositions of $w_0$. For a reduced expression $\underline{w}_0=s_{i_1}\cdots s_{i_N}$ where $N=\#\Delta_+$, we denote $\beta_t=s_{i_1}\cdots s_{i_{t-1}}(\alpha_{i_t})$ for $1\leq t\leq N$; then $\Delta_+=\{\beta_t\mid\ t=1,\cdots,N\}$. 
\par
Let $U_q(\g)$ be the generic quantum group associated to $\g$. It is generated as an algebra by $E_i$, $F_i$ and $K_i^{\pm 1}$ for $i=1,\cdots,n$, subject to the following relations: for $i,j=1,\cdots,n$,
$$K_iK_i^{-1}=K_i^{-1}K_i=1,\ \ K_iE_jK_i^{-1}=q_i^{a_{ij}}E_j,\ \ K_iF_jK_i^{-1}=q_i^{-a_{ij}}F_j,$$
$$E_iF_j-F_jE_i=\delta_{ij}\frac{K_i-K_i^{-1}}{q_i-q_i^{-1}},$$
and for $i\neq j$,
$$\sum_{r=0}^{1-a_{ij}}(-1)^r E_i^{(1-a_{ij}-r)}E_jE_i^{(r)}=0,\ \ \sum_{r=0}^{1-a_{ij}}(-1)^rF_i^{(1-a_{ij}-r)}F_jF_i^{(r)}=0,$$
where 
\[
q_i=q^{d_i},\, \, [n]_q!=\prod_{i=1}^n \frac{q^n-q^{-n}}{q-q^{-1}},\  \ E_i^{(n)}=\frac{E_i^n}{[n]_{q_i}!}\ \ \text{and}\ \ F_i^{(n)}=\frac{F_i^n}{[n]_{q_i}!}.
\]

Let $U_q(\mathfrak{n}^-)$ be the sub-algebra of $U_q(\g)$ generated by $F_i$ for $i=1,\cdots,n$. 
\par
Let $T_i=T_{i,1}''$, $i=1,\cdots,n$ be Lusztig's automorphisms:
$$T_i(E_i)=-F_iK_i,\ \ T_i(F_i)=-K_i^{-1}E_i,\ \ T_i(K_j)=K_jK_i^{-a_{ij}},$$
for $i=1,\cdots,n$, and for $j\neq i$,
$$T_i(E_j)=\sum_{r+s=-a_{ij}}(-1)^rq_i^{-r}E_i^{(s)}E_jE_i^{(r)},\ \ T_i(F_j)=\sum_{r+s=-a_{ij}}(-1)^rq_i^{r}F_i^{(r)}F_jF_i^{(s)}.$$
Details can be found in Chapter 37 of \cite{Lus}. The quantum PBW root vector $F_{\beta_t}$ associated to $\beta_t$ is defined by:
$$F_{\beta_t}:=T_{i_1}T_{i_2}\cdots T_{i_{t-1}}(F_{i_t})\in U_q(\n^-).$$
For $s\in\mathbb{N}$, we define 
$$F_{\beta_t}^{(s)}:=\frac{F_{\beta_t}^{s}}{[s]_{q_{i_t}}!}.$$
The PBW theorem of quantum groups says that the set
$$\{F_{\beta_1}^{(c_1)}F_{\beta_2}^{(c_2)}\cdots F_{\beta_N}^{(c_N)}\mid (c_1,\cdots,c_N)\in\mathbb{N}^N\}$$
is a linear basis of $U_q(\mathfrak{n}^-)$ (\cite{Lus}, Corollary 40.2.2).

\subsection{Commutation relations of quantum PBW root vectors}
The commutation relations of these quantum PBW root vectors are described in the following Levendorskii-Soibelman (LS for short) formula:

\begin{theorem}[\cite{LS1, dCK}]
For any $1\leq k<\ell\leq N$,
\begin{equation}\label{Eq:LS}
F_{\beta_\ell}F_{\beta_k}-q^{-(\beta_k,\beta_\ell)}F_{\beta_k}F_{\beta_\ell}=\sum_{n_{k+1},\ldots,n_{\ell-1}\geq 0}c(n_{k+1},\ldots,n_{\ell-1})F_{\beta_{k+1}}^{(n_{k+1})}\ldots F_{\beta_{\ell-1}}^{(n_{\ell-1})},
\end{equation}
where $c(n_{k+1},\ldots,n_{\ell-1})\in\mathbb{Z}[q^{\pm 1}]$.
\end{theorem}

A more general result is proved by Xi in \cite{Xi}:
\begin{theorem}[\cite{Xi}]\label{Thm:Xi}
For any $1\leq k<\ell\leq N$ and $s,t\in\mathbb{N}$, 
\begin{equation}\label{Eq:Xi}
F_{\beta_\ell}^{(s)}F_{\beta_k}^{(t)}-q^{-st(\beta_k,\beta_\ell)}F_{\beta_k}^{(t)}F_{\beta_\ell}^{(s)}=\sum_{n_{k},\ldots,n_{\ell}\geq 0,\ n_k<t,\ n_\ell<s}c(n_{k},\ldots,n_{\ell})F_{\beta_{k}}^{(n_{k})}\ldots F_{\beta_{\ell}}^{(n_{\ell})},
\end{equation}
where $c(n_{k},\ldots,n_{\ell})\in\mathbb{Z}[q^{\pm 1}]$.
\end{theorem}

Since the total order $\beta_1<\beta_2<\ldots<\beta_N$ is convex, by specializing $q$ to $1$, we know that in the LS-formula \eqref{Eq:LS}, if there exists $k<r<\ell$ such that $\beta_r=\beta_k+\beta_\ell$, then the monomial $F_{\beta_r}$ appears on the right hand side of \eqref{Eq:LS} with non-zero coefficient.
\par
These are all information that can be read from the specialization $q\to 1$. It is natural to ask: what are the monomials appearing on the right hand side of \eqref{Eq:LS} with non-zero coefficients?
\par
The goal of the rest of this section is to show that particular monomials may appear in some situations. This will be applied later to study the quantum degree cones.
\par
For $1\leq k\leq N$ and $s\in\mathbb{N}$, we define $k[s]:=t$ if $t>k$ such that $i_t=i_k$ and the number of $i_k$ appearing in $\{i_{k+1},\ldots,i_t\}$ is exactly $s$; otherwise we set it to be zero. Hence $k[1]=\ell$ implies that $i_k=i_\ell$ and for any $k<p<\ell$, $i_p\neq i_k$.

\begin{theorem}\label{Thm:term}
Suppose that $k[1]=\ell$ and $c_s=-\langle\alpha_{i_k},\alpha_{i_s}^\vee\rangle$. The monomial $F_{\beta_{k+1}}^{(c_{k+1})}\ldots F_{\beta_{\ell-1}}^{(c_{\ell-1})}$ appears in $F_{\beta_\ell}F_{\beta_k}-q^{-(\beta_k,\beta_\ell)}F_{\beta_k}F_{\beta_\ell}$ with non-zero coefficient.
\end{theorem}

The rest of this section will be devoted to prove this theorem.
\par
First observe that by applying $T_{i_{k-1}}^{-1}\ldots T_{i_1}^{-1}$, we can suppose that $k=1$. By assumption, $i_\ell=i_1$. 

\subsection{Preparations to the proof}
We keep the notations in the statement of the theorem and start deriving some combinatorial formulas on roots.

\begin{lemma}\label{Lem:Root}
The following statements hold:
\begin{enumerate}[(1).]
\item We have $\beta_1 + \beta_{\ell} = \sum_{j=2}^{\ell-1} c_j \beta_{j}.$
\item The expression $s_{i_1}\ldots s_{i_{\ell-2}}s_{i_1}$ is reduced.
\item When $c_2\neq 0$, the root $s_{i_1}\ldots s_{i_{\ell-2}}(\alpha_{i_1})\in\Delta_+$.
\end{enumerate}
\end{lemma}

\begin{proof}
\begin{enumerate}
\item By the definition of $\beta_\ell$, we have:
\begin{eqnarray*}
\beta_\ell&=&s_{i_1}\ldots s_{i_{\ell-2}}s_{i_{\ell-1}}(\alpha_{i_1})\\
&=& s_{i_1}\ldots s_{i_{\ell-2}}(\alpha_{i_1}+c_{\ell-1}\alpha_{i_{\ell-1}})\\
&=& s_{i_1}\ldots s_{i_{\ell-2}}(\alpha_{i_1})+c_{\ell-1}\beta_{\ell-1}
\end{eqnarray*}
Iterating this computation gives 
$$\beta_\ell=s_{i_1}(\alpha_{i_1})+c_2\beta_2+\ldots+c_{\ell-1}\beta_{\ell-1}.$$
As $\beta_1=\alpha_{i_1}$, the formula is proved.
\item It suffices to notice that there is no $i_1$ in the set $\{i_2,\ldots,i_{\ell-2}\}$, and the expression $s_{i_1}\ldots s_{i_{\ell-2}}$ is reduced.
\item The same argument as in (1) shows that
$$s_{i_1}\ldots s_{i_{\ell-2}}(\alpha_{i_1})=-\alpha_{i_1}+c_2\beta_2+\ldots+c_{\ell-2}\beta_{\ell-2}.$$
Since for any $2\leq s\leq \ell-1$, $c_s\geq 0$, it suffices to show that $c_2\beta_2-\alpha_{i_1}$ is a positive root. We have
$$c_2\beta_2-\alpha_{i_1}=c_2\alpha_{i_2}-(a_{i_1,i_2}+1)\alpha_{i_1}.$$
As $c_2=-\langle\alpha_{i_1},\alpha_{i_2}^\vee\rangle\neq 0$, $-a_{i_1,i_2}=-\langle\alpha_{i_2},\alpha_{i_1}^\vee\rangle\geq 1$, this proves the statement.
\end{enumerate}
\end{proof}

The following lemma will be used later.

\begin{lemma}\label{Lem:X}
Suppose that $c_2\neq 0$ and let $X:=T_{i_1}\ldots T_{i_{\ell-2}}(F_{i_1})$. 
\begin{enumerate}
\item We have $X\in U_q(\lie n^-)$.
\item In the following expression of $X$:
$$X=\sum_{n_1,\ldots,n_N\geq 0}h(n_1,\ldots,n_N)F_{\beta_1}^{(n_1)}\ldots F_{\beta_N}^{(n_N)},$$
if $h(n_1,\ldots,n_N)\neq 0$, then $n_{\ell-1}=n_{\ell}=\ldots=n_N=0$.
\end{enumerate}
\end{lemma}

\begin{proof}
\begin{enumerate}
\item The statement holds for weight reason and Lemma \ref{Lem:Root} (3).
\item Since $c_2\neq 0$, there exists $2\leq t\leq\ell-2$ such that $c_t\neq 0$ and for any $t<p\leq \ell-2$, $c_p=0$. Then
\begin{eqnarray*}
X &=& T_{i_1}\ldots T_{i_t}(F_{i_1})\\
&=& \sum_{r+s=c_t}(-1)^r q_{i_t}^r T_{i_1}\ldots T_{i_{t-1}}(F_{i_t}^{(r)})T_{i_1}\ldots T_{i_{t-1}}(F_{i_1})T_{i_1}\ldots T_{i_{t-1}}(F_{i_t}^{(s)})\\ 
&=& \sum_{r+s=c_t}(-1)^r q_{i_t}^r F_{\beta_t}^{(r)}T_{i_1}\ldots T_{i_{t-1}}(F_{i_1}) F_{\beta_t}^{(s)}.
\end{eqnarray*}
Since $t-1<\ell-2$, iterating this procedure and applying Theorem \ref{Thm:Xi}  to write the quantum PBW root vectors into correct order, we obtained that $n_{t+1}=n_{t+2}=\ldots=n_N=0$. As $t\leq\ell-2$, $n_{\ell-1}=n_\ell=\ldots=n_N=0$.
\end{enumerate}
\end{proof}

\subsection{Proof of Theorem \ref{Thm:term}: rank 2 cases}
When the rank of $\g$ is $2$, to prove the theorem, it suffices to consider the cases $(i_1,i_2,i_3)=(1,2,1)$ and $(i_1,i_2,i_3)=(2,1,2)$.
\par
We first prove the theorem for $(i_1,i_2,i_3)=(2,1,2)$ in $\tt A_2$, $\tt B_2$ and $\tt C_2$, the other case can be proved similarly.

\begin{enumerate}
\item Case $\tt A_2$. In this case $\beta_1=\alpha_2$, $\beta_3=\alpha_1$, hence $F_{\beta_3}=F_1$. The commutation relation reads $F_1F_2-qF_2F_1$, which is nothing but $T_2(F_1)$.
\item Case $\tt B_2$. In this case $\beta_1=\alpha_2$, $\beta_3=2\alpha_1+\alpha_2$. The commutation relation is
$$F_{\beta_3}F_{\beta_1}-F_{\beta_1}F_{\beta_3}=(-1+q^{-2})F_{\beta_2}^{(2)}.$$
\item Case $\tt C_2$. In this case $\beta_1=\alpha_2$, $\beta_3=\alpha_1+\alpha_2$. The commutation relation is given by
$$F_{\beta_3}F_{\beta_1}-F_{\beta_1}F_{\beta_3}=(q+q^{-1})F_{\beta_2}.$$
\end{enumerate}

It remains to consider the $\tt G_2$ case. For $(i_1,i_2,i_3)=(1,2,1)$, $\beta_1=\alpha_1$ and $\beta_3=2\alpha_1+3\alpha_2$. 
The commutation relation reads
$$F_{\beta_3}F_{\beta_1}-q^{-3}F_{\beta_1}F_{\beta_3}=(1-q^{-2}-q^{-4}+q^{-6})F_{\beta_2}^{(3)}.$$
For $(i_1,i_2,i_3)=(2,1,2)$, $\beta_1=\alpha_2$ and $\beta_3=\alpha_1+2\alpha_2$. 
The commutation relation is
$$F_{\beta_3}F_{\beta_1}-q^{-1}F_{\beta_1}F_{\beta_3}=(q^2+1+q^{-2})F_{\beta_2}.$$

\subsection{Proof of Theorem \ref{Thm:term}: general case}
The proof is executed by induction on $\ell$. The case $\ell=3$ is the rank $2$ case proved above.
\par
Moreover, we can assume that $c_2\neq 0$ and $c_{\ell-1}\neq 0$, since otherwise the statement holds by induction hypothesis.
\par
We start with considering the quantum PBW root vector $F_{\beta_\ell}$:
\begin{eqnarray*}
F_{\beta_\ell} &=& T_{i_1}\ldots T_{i_{\ell-1}}(F_{i_1})\\
&=& \sum_{r+s=c_{\ell-1}}(-1)^r q_{i_{\ell-1}}^r T_{i_1}\ldots T_{i_{\ell-2}}(F_{i_{\ell-1}}^{(r)}F_{i_1}F_{i_{\ell-1}}^{(s)})\\
&=& \sum_{r+s=c_{\ell-1}}(-1)^r q_{i_{\ell-1}}^r F_{\beta_{\ell-1}}^{(r)}T_{i_1}\ldots T_{i_{\ell-2}}(F_{i_1})F_{\beta_{\ell-1}}^{(s)}.
\end{eqnarray*}
We suppose that $X=T_{i_1}\ldots T_{i_{\ell-2}}(F_{i_1})$ and $\gamma=s_{i_1}\ldots s_{i_{\ell-2}}(\alpha_{i_1})$. The commutator of quantum PBW root vectors gives
$$F_{\beta_\ell}F_{\beta_1}-q^{-(\beta_1,\beta_\ell)}F_{\beta_1}F_{\beta_\ell}=\sum_{r+s=c_{\ell-1}}(-1)^r q_{i_{\ell-1}}^r \left(F_{\beta_{\ell-1}}^{(r)}XF_{\beta_{\ell-1}}^{(s)}F_{\beta_1}-q^{-(\beta_1,\beta_\ell)}F_{\beta_1}F_{\beta_{\ell-1}}^{(r)}XF_{\beta_{\ell-1}}^{(s)}\right).$$
By Theorem \ref{Thm:Xi}, we have 
$$F_{\beta_{\ell-1}}^{(r)}XF_{\beta_{\ell-1}}^{(s)}F_{\beta_1}=q^{-s(\beta_1,\beta_{\ell-1})}F_{\beta_{\ell-1}}^{(r)}XF_{\beta_1}F_{\beta_{\ell-1}}^{(s)}+\Sigma_1.$$
We claim that $F_{\beta_2}^{(c_2)}\ldots F_{\beta_{\ell-1}}^{(c_{\ell-1})}$ does not appear with non-zero coefficient in $\Sigma_1$. Indeed, by Lemma \ref{Lem:X}, writing $X$ in the PBW basis, $F_{\beta_{\ell-1}}$ does not appear in any monomial. Hence commuting $F_{\beta_{\ell-1}}^{(r)}$ with $X$ using Theorem \ref{Thm:Xi} will not producing any element in the PBW basis with the power of $F_{\beta_{\ell-1}}$ higher than $r$. Again by Theorem \ref{Thm:Xi}, commuting $F_{\beta_{\ell-1}}^{(s)}$ and $F_{\beta_1}$ will produce elements in the PBW basis with the power of $F_{\beta_{\ell-1}}$ strictly less than $s$. Combining them together proves the claim.
\par
The same argument shows that, if we write using Theorem \ref{Thm:Xi}
$$F_{\beta_1}F_{\beta_{\ell-1}}^{(r)}XF_{\beta_{\ell-1}}^{(s)}=q^{-r(\beta_1,\beta_{\ell-1})}F_{\beta_{\ell-1}}^{(r)}F_{\beta_1}XF_{\beta_{\ell-1}}^{(s)}+\Sigma_2,$$
then the monomial $F_{\beta_2}^{(c_2)}\ldots F_{\beta_{\ell-1}}^{(c_{\ell-1})}$ does not appear in $\Sigma_2$ with non-zero coefficient.
\par
As a consequence, to study the coefficient of $F_{\beta_2}^{(c_2)}\ldots F_{\beta_{\ell-1}}^{(c_{\ell-1})}$, it suffices to consider the term
\begin{eqnarray*}
& &\sum_{r+s=c_{\ell-1}}(-1)^r q_{i_{\ell-1}}^r \left(q^{-s(\beta_1,\beta_{\ell-1})}F_{\beta_{\ell-1}}^{(r)}XF_{\beta_1}F_{\beta_{\ell-1}}^{(s)}-q^{r(\beta_1,\beta_{\ell-1})-(\beta_1,\beta_\ell)}F_{\beta_{\ell-1}}^{(r)}F_{\beta_1}XF_{\beta_{\ell-1}}^{(s)}\right)\\
&=&\sum_{r+s=c_{\ell-1}}(-1)^r q_{i_{\ell-1}}^r q^{-s(\beta_1,\beta_{\ell-1})}F_{\beta_{\ell-1}}^{(r)}\left(XF_{\beta_1}-q^{c_{\ell-1}(\beta_1,\beta_{\ell-1})-(\beta_1,\beta_\ell)}F_{\beta_1}X\right)F_{\beta_{\ell-1}}^{(s)}.
\end{eqnarray*}
We simplify the power of $q$ in the bracket: by the proof of Lemma \ref{Lem:Root} (1), $$c_{\ell-1}(\beta_1,\beta_{\ell-1})-(\beta_1,\beta_\ell)=-(\beta_1,\beta_\ell-c_{\ell-1}\beta_{\ell-1})=-(\beta_1,\gamma).$$
The term in the bracket now reads
$$T_{i_1}\ldots T_{i_{\ell-2}}(F_{i_1})F_{\beta_1}-q^{-(\beta_1,\gamma)}F_{\beta_1}T_{i_1}\ldots T_{i_{\ell-2}}(F_{i_1}).$$
By induction hypothesis, in this commutator, when written in the PBW basis, there exists a term $\lambda F_{\beta_2}^{(c_2)}\ldots F_{\beta_{\ell-2}}^{(c_{\ell-2})}$ with $\lambda\neq 0$. From this we obtain the following term in the above sum
$$\sum_{r+s=c_{\ell-1}}(-1)^r q_{i_{\ell-1}}^r q^{-s(\beta_1,\beta_{\ell-1})}\lambda F_{\beta_{\ell-1}}^{(r)} F_{\beta_2}^{(c_2)}\ldots F_{\beta_{\ell-2}}^{(c_{\ell-2})}F_{\beta_{\ell-1}}^{(s)}.$$
By applying Theorem \ref{Thm:Xi} again, we obtain
\begin{equation}\label{Eq:E1}
F_{\beta_{\ell-1}}^{(r)} F_{\beta_2}^{(c_2)}\ldots F_{\beta_{\ell-2}}^{(c_{\ell-2})}-q^{-\mu} F_{\beta_2}^{(c_2)}\ldots F_{\beta_{\ell-2}}^{(c_{\ell-2})}F_{\beta_{\ell-1}}^{(r)}=\Sigma_3,
\end{equation}
where $\mu=\sum_{t=2}^{\ell-2}c_t r(\beta_t,\beta_{\ell-1})$, and the monomial $F_{\beta_2}^{(c_2)}\ldots F_{\beta_{\ell-2}}^{(c_{\ell-2})}F_{\beta_{\ell-1}}^{(r)}$ does not appear in $\Sigma_3$ with non-zero coefficient. Therefore the monomial $F_{\beta_2}^{(c_2)}\ldots F_{\beta_{\ell-2}}^{(c_{\ell-2})}F_{\beta_{\ell-1}}^{(c_{\ell-1})}$ appears in \eqref{Eq:E1} with coefficient
\begin{equation}\label{Eq:E2}
\lambda \sum_{r+s=c_{\ell-1}}(-1)^r q_{i_{\ell-1}}^r q^{-s(\beta_1,\beta_{\ell-1})-\sum_{t=2}^{\ell-2}c_t r(\beta_t,\beta_{\ell-1})}\bino{c_{\ell-1}}{q}_{q_{i_{\ell-1}}}.
\end{equation}
It suffices to show that this term is non-zero. We first simplify the power of $q$: apply Lemma \ref{Lem:Root} (1) and replace $s=c_{\ell-1}-r$,
\begin{eqnarray*}
s(\beta_1,\beta_{\ell-1})+\sum_{t=2}^{\ell-2}c_t r(\beta_t,\beta_{\ell-1})&=&c_{\ell-1}(\beta_1,\beta_{\ell-1})+r(-\beta_1+\sum_{t=2}^{\ell-2}c_t\beta_t,\beta_{\ell-1})\\
&=& c_{\ell-1}(\beta_1,\beta_{\ell-1})+r(\beta_\ell-c_{\ell-1}\beta_{\ell-1},\beta_{\ell-1})\\
&=&c_{\ell-1}(\beta_1,\beta_{\ell-1})+r(s_{i_{\ell-1}}(\alpha_{i_1})-c_{\ell-1}\alpha_{i_{\ell-1}},\alpha_{i_{\ell-1}})\\
&=& c_{\ell-1}(\beta_1,\beta_{\ell-1})+r(\alpha_{i_1},\alpha_{i_{\ell-1}}),
\end{eqnarray*}
where for the third equality, the $W$-invariance is used. Since both $\lambda$ and $q^{-(\beta_1,\beta_{\ell-1})}$ are non-zero constants, to show that the coefficient \eqref{Eq:E2} is non-zero, it suffices to show that
\begin{equation}\label{Eq:E3}
\sum_{r+s=c_{\ell-1}}(-1)^r q_{i_{\ell-1}}^r q^{-r(\alpha_{i_1},\alpha_{i_{\ell-1}})}\bino{c_{\ell-1}}{q}_{q_{i_{\ell-1}}}=\sum_{r+s=c_{\ell-1}}(-1)^r q_{i_{\ell-1}}^{r(1+c_{\ell-1})}\bino{c_{\ell-1}}{q}_{q_{i_{\ell-1}}}\neq 0.
\end{equation}
Here we used the fact that $c_{\ell-1}=-d_{i_{\ell-1}}^{-1}(\alpha_{i_{\ell-1}},\alpha_{i_1})$.
\par
By the assumption, $c_{\ell-1}\neq 0$, hence in the case where the rank of $\g$ is higher than $2$, $c_{\ell-1}=1$ or $2$. If $c_{\ell-1}=1$, the equation \eqref{Eq:E3} reads $1-q_{i_{\ell-1}}^2$. If $c_{\ell-1}=2$, the equation \eqref{Eq:E3} reads $1-q_{i_{\ell-1}}^2-q_{i_{\ell-1}}^4+q_{i_{\ell-1}}^6$. In both cases, they are non-zero. This terminates the proof of Theorem \ref{Thm:term}.

\begin{remark}
It is clear from the proof that Theorem \ref{Thm:term} still holds when $q$ is not a small root of unity. 
\end{remark}

\section{Quantum degree cone and Lusztig tight monomial cone}\label{s3}

\subsection{Lusztig tight monomial cones}

In \cite{Lus2} Lusztig defined a polyhedral cone in the simply-laced cases in order to study monomials in the canonical basis. The natural generalisations to all Lie types are given in \cite{CMMG}.

Let $\underline{w}_0=s_{i_1}\cdots s_{i_N}$ be a reduced decomposition of $w_0$. The Lusztig tight monomial cone $\mathcal{L}_{\underline{w}_0}(\g)\subset\mathbb{R}^N$ is the polyhedral cone defined by the following inequalities:
\begin{enumerate}
\item[(i)] for any two indices $1\leq p<p'\leq N$ with $p'=p[1]$, we have
$$x_p+x_{p'}+\sum_{p<s<p'}a_{i_p,i_s}x_s\leq 0;$$
\item[(ii)] for $1\leq p\leq N$, $x_p\geq 0$.
\end{enumerate}

It is shown in \cite{CMMG} that the cone $\mathcal{L}_{\underline{w}_0}(\g)$ is simplicial..

We will denote $\mathcal{L}_{\underline{w}_0}^-(\g)\subset\mathbb{R}^N$, termed a \emph{negative tight monomial cone}, the polyhedral cone defined by 
\begin{enumerate}
\item[(i)'] for any two indices $1\leq p<p'\leq N$ with $p'=p[1]$, we have
$$x_p+x_{p'}+\sum_{p<s<p'}a_{i_p,i_s}x_s\geq 0.$$
\end{enumerate}

The cone $\mathcal{L}_{\underline{w}_0}^-(\g)$ is a product of a linearity space of dimension $n$ and a simplicial cone of dimension $N-n$ by setting the coordinates associated to simple roots to zero.

\subsection{Quantum degree cones}

The quantum degree cones  are defined in \cite{BFF}, motivated by studying the quantum PBW filtration on quantum groups.
\par
We keep the notations in LS-formula. The quantum degree cone $\mathcal{D}_{\underline{w}_0}^q(\g)$ associated to a reduced decomposition $\underline{w}_0$ is defined by:
$$\mathcal{D}_{\underline{w}_0}^q(\g):=\{\mathbf{d}=(d_\beta)_{\beta\in\Delta_+}\in\mathbb{R}^{\Delta_+}\mid \text{ for }i<j,\ d_{\beta_i}+d_{\beta_j}\geq \sum_{k=i+1}^{j-1}n_k d_{\beta_k}\text{ if }c(n_{i+1},\cdots,n_{j-1})\neq 0\}.$$

We will denote $\overline{\mathcal{D}}_{\underline{w}_0}^q(\g):=\mathcal{D}_{\underline{w}_0}^q(\g)\cap\mathbb{R}_{\geq 0}^{\Delta_+}$ its non-negative part. Our notation here is slightly different to \cite{BFF}: the quantum degree cone therein is the interior of the non-negative cone $\overline{\mathcal{D}}_{\underline{w}_0}^q(\g)$.

It is proved in \textit{loc.cit.} that $\mathcal{D}_{\underline{w}_0}^q(\g)$ is non-empty, and any $\mathbf{d}\in\mathcal{D}_{\underline{w}_0}^q(\g)$ defines a filtered algebra structure on $U_q(\mathfrak{n}^-)$ by letting $\deg(F_{\beta_i})=d_{\beta_i}$. The associated graded algebra is isomorphic to a skew-polynomial algebra if and only if $\mathbf{d}$ is contained in the interior of $\mathcal{D}_{\underline{w}_0}^q(\g)$.

\subsection{Relations between two cones}

For a fixed $\underline{w}_0$, we identify $\mathbb{R}^{\Delta_+}$ with $\mathbb{R}^N$ by sending $(d_\beta)_{\beta\in\Delta_+}$ to $(d_{\beta_1},d_{\beta_2},\cdots,d_{\beta_N})$.

As a consequence of Theorem \ref{Thm:term}, for a fixed $\underline{w}_0$, we obtain an embedding of the quantum degree cone to the negative tight monomial cone. We denote $\g^L$ the Langlands dual of the Lie algebra $\g$.

\begin{corollary}\label{Cor:subset}
For any $\underline{w}_0\in R(w_0)$, we have $\mathcal{D}_{\underline{w}_0}^q(\g)\subseteq\mathcal{L}_{\underline{w}_0}^-(\g^L)$.
\end{corollary}

We conjecture that these two cones coincide.

\begin{conjecture}\label{conjmain}
For any $\underline{w}_0\in R(w_0)$, we have $\mathcal{D}_{\underline{w}_0}^q(\g)=\mathcal{L}_{\underline{w}_0}^-(\g^L)$.
\end{conjecture}

In the next section we prove this conjecture in the case where $\g$ is simply-laced and $\underline{w}_0$ is compatible with a quiver.

\section{Quiver representations and Hall algebras}\label{s4}

A basic reference on quiver representations and Auslander-Reiten theory is \cite{ASS}. For Hall algebras, we refer to \cite{Rin90}, \cite{FFR16}.

\subsection{Quiver representations}
Let $\mathbb{K}$ be a field. We fix a Dynkin quiver of type $\tt A$, $\tt D$, or $\tt E$ with vertices $I$ and the number of arrows between two different vertices $i$ and $j$ (in either direction) equals $-a_{ij}$.
\par
Let $\mathbb{K}Q$ be the path algebra of $Q$, $\mathrm{mod}(\mathbb{K}Q)$ be the category of finite dimensional $\mathbb{K}Q$-modules. 
\par
We denote $K_0:=K_0(\mathbb{K}Q)$ the Grothendieck group of $\mathbb{K}Q$, which can be identified with $\mathbb{Z}^I$ by mapping the class of a module to its dimension vector. 
\par
Let $K_0^\oplus:=K_0(\mathrm{mod}(\mathbb{K}Q))$ be the split Grothendieck group of $\mathbb{K}Q$-modules modulo split exact sequences. That is to say, for $M,N\in\mathrm{mod}(\mathbb{K}Q)$, $[M\oplus N]=[M]+[N]$. The abelian group $K_0^\oplus$ is freely generated by the isomorphism classes of indecomposable modules.
\par
The simple modules $S_i$ in $\mathrm{mod}(\mathbb{K}Q)$ are parametrized by the vertices $i\in I$. This identifies $K_0=\mathbb{Z}^I$ with the root lattice of $\g$, the Lie algebra having the underlying graph of $Q$ as Dynkin diagram, by sending the dimension vector of $S_i$ to the simple root $\alpha_i$. The indecomposable modules are parametrized by the positive roots in the root lattice: for each $\alpha\in\Delta_+$, there exists a unique (up to isomorphism) indecomposable module $U_\alpha$ having dimension vector $\alpha$.
\par
Notice that taking the dimension vector induces a map $\mathbf{dim}:K_0^\oplus\ra K_0$.
\par
For two $\mathbb{K}Q$-modules $M$ and $N$, we will denote
$$[M,N]:=\dim\Hom(M,N)\text{ and }[M,N]^1=\dim\Ext^1(M,N).$$

We consider the category ${\rep}_{\mk}Q$ of finite dimensional $\mk$-representations of $Q$, which is an abelian $\mk$-linear category of global dimension at most one. We denote by $\mathbf{dim}M$ the dimension vector of a representation $M$, viewed as an element of the root lattice via 
$$\mathbf{dim}M = \sum_{i\in I} \dim(M_i)\alpha_i.$$ 
The Euler form of $Q$ is defined by:
$$\langle\mathbf{dim}M,\mathbf{dim}N\rangle = [M,N]-[M,N]^1.$$ 
We denote its symmetrization by $(-,-)$, which is the bilinear form defined by the Cartan matrix.

\subsection{Auslander-Reiten quivers arising from reduced decompositions}

The category $\mathrm{mod}(\mk Q)$ is representation-directed, which means that there exists an enumeration $\beta_1,\cdots,\beta_N$ of the positive roots such that
$$\Hom(U_{\beta_k},U_{\beta_l})=0\mbox{ for }k>l\mbox{ and }{\Ext}^1(U_{\beta_k},U_{\beta_l})=0\mbox{ for }k\leq l.$$ 
We will write $U_k:=U_{\beta_k}$ for short.
It is known that there exists a sequence $i_1,\cdots,i_N$ in $I$ such that $\underline{w}_0:=s_{i_1}\cdots s_{i_N}$ is a reduced decomposition in $R(w_0)$, and such that $\beta_k = s_{i_1}\cdots s_{i_{k-1}}(\alpha_{i_k})$ for all $k = 1,\cdots,N$.
\par
Many representation theoretical information can be recovered from this reduced decomposition.
\begin{enumerate}
\item The Auslander-Reiten quiver can be constructed in the following way:
\begin{itemize}
\item[-] the vertices are $1,2,\cdots,N$;
\item[-] there are arrows $k\ra\ell$ and $\ell\to k[1]$ (when $k[1]$ exists)  if there exists an arrow $i_\ell\ra i_k$ in $Q$, and $\ell$ is minimal along those indices larger then $k$ with this property;
\item[-] there is a translation $i_k\dashrightarrow i_\ell$ if $\ell=k[1]$.
\end{itemize}
\item for $k\leq\ell$, the dimension $[U_k,U_\ell]$ equals $\langle U_k,U_\ell\rangle$;
\item for $k\leq\ell$, the dimension $[U_\ell,U_k]^1$ equals $-\langle U_\ell,U_k\rangle$.
\end{enumerate}

We define a partial order $U\preceq V$ on the indecomposable modules if there exists a path from $U$ to $V$ in the Auslander-Reiten quiver. By the above argument, this partial order is compatible with the directed enumeration, that is to say, if $U_k\preceq U_\ell$ then $k\leq\ell$.

\subsection{Hall algebras} 
We consider the (generic) Hall algebra $\mathcal{H}(Q)$ of $Q$. It is the $\mathbb{Q}(q)$-algebra with basis $F_{[V]}$ for $V$ ranging over all isomorphism classes of finite dimensional $\mathbb{K}Q$-modules. The multiplication is defined by:
$$F_{[V]}F_{[W]}=\sum_{[X]}q^{[V,V]+[W,W]+\langle\mathbf{dim}V,\mathbf{dim}W\rangle-[X,X]}H_{V,W}^X(q^2) F_{[X]},$$
where $H_{V,W}^X$ is the Hall polynomial counting the number of sub-modules $U$ of $X$ which are isomorphic to $W$, with quotient $X/U$ isomorphic to $V$ over $\mathbb{F}_q$. 
\par
We summarises some basic facts about these Hall polynomials:
\begin{enumerate}
\item $H_{V,W}^X(q)\neq 0$ if and only if there exists a short exact sequence $0\ra W\ra X\ra V\ra 0$.
\item If $[V,W]=0$ and $X$ is an extension of $V$ by $W$ then $H_{V,W}^X(q)=1$.
\item If $[W,V]=0$ then $H_{W,V}^{W\oplus V}(q)=q^{[V,W]}$.
\end{enumerate}

The algebra $\mathcal{H}(Q)$ is isomorphic to $U_q(\mathfrak{n}^-)$ \cite{Rin90}. The isomorphism is given by $F_{[S_i]}\mapsto F_i$. For an indecomposable module $U$, we denote the image of $F_{[U]}$ by $F_{\mathbf{dim}\,U}$.
\par
Given a finite dimensional $\mathbb{K}Q$-module $M$, we can decompose it as 
$$M=\bigoplus_{k=1}^N U_k^{m_k}.$$
Then in the Hall algebra we have 
$$F_{[M]}=F_{[U_1]}^{(m_1)}F_{[U_2]}^{(m_2)}\cdots F_{[U_N]}^{(m_N)}.$$

For two indecomposable modules $U\preceq V$, their $q$-commutator is given by:
$$F_{[V]}F_{[U]}-q^{[U,V]-[V,U]^1}F_{[U]}F_{[V]}=\text{terms coming from non-split extensions of $V$ by $U$}.$$

We define the $q$-commutator by:
$$[F_{[V]},F_{[U]}]_q:=F_{[V]}F_{[U]}-q^{[U,V]-[V,U]^1}F_{[U]}F_{[V]}.$$
\begin{enumerate}
\item One special case is when $[V,U]^1=0$ then $[F_{[V]},F_{[U]}]_q=0$.
\item The other special case is when $U=\tau V$ is the Auslander-Reiten translation of $V$. In this case there exists exactly one non-split extension $X$ of $V$ by $U$ given by the middle term of the Auslander-Reiten sequence, and the $q$-commutator can be computed as
$$[F_{[V]},F_{[U]}]_q=(q-q^{-1})^{[U,V]}F_{[X]}.$$
\end{enumerate}
We translate these results in $\mathcal{H}(Q)$ to $U_q(\mathfrak{n}^-)$ using the algebra isomorphism described above. Under this isomorphism, $F_{[U_k]}$ is sent to the quantum PBW root vector $F_{\beta_k}$ corresponding to the reduced decomposition arising from the enumeration of positive roots. Then the above formulae can be rewritten as
\begin{enumerate}
\item If $\langle\beta_\ell,\beta_k\rangle=0$, then $F_{\beta_\ell}F_{\beta_k}=F_{\beta_k}F_{\beta_\ell}$.
\item If $1\leq k<\ell\leq N$ and $\ell=k[1]$, then
$$[F_{\beta_\ell},F_{\beta_k}]_q=(q-q^{-1})^{\langle\beta_k,\beta_\ell\rangle}\prod_{k<p<\ell,\ \langle\alpha_{i_p},\alpha_{i_k}\rangle\neq 0}F_{\beta_p}.$$
\end{enumerate}

We describe the terms appearing in the $q$-commutator.

\begin{lemma}\label{Lem:degenerate}
The following statements are equivalent:
\begin{enumerate}
\item a term $F_{[X]}$ appears in $[F_{[V]},F_{[U]}]_q$;
\item $X$ is a non-split extension of $V$ by $U$;
\item $X$ properly degenerates to $U\oplus V$;
\item for all indecomposable modules $Z$, we have $[Z,X]\leq [Z,U]+[Z,V]$, with strict inequality for some $Z$.
\end{enumerate}
\end{lemma}

\begin{proof}
The first two statements are equivalent by the above argument. The second and third statements are equivalent by \cite[Theorem 4.5 (a)]{B}. The third and the fourth statements are equivalent by \cite[Proposition 3.2]{B}.
\end{proof}

\begin{proposition}\label{Prop:degenerate}
A term $F_{[X]}$ appears in $[F_{[V]},F_{[U]}]_q$ if and only if the following conditions hold:
\begin{enumerate}
\item $\mathbf{dim} X=\mathbf{dim} U+\mathbf{dim}V$;
\item $U\prec Y\prec V$ for every indecomposable direct summand $Y$ of $X$;
\item $[Z,X]\leq [Z,V]$ for all indecomposable modules $Z$ such that $\tau^{-1}U'\preceq Z\prec V$, with strict inequality for at least one such $Z$.
\end{enumerate}
\end{proposition}

\begin{proof}
By Lemma \ref{Lem:degenerate}, the first and the third conditions are necessary is clear. For the second condition, note that every such $Y$ has to admit a non-zero map from $U$ and to $V$, otherwise it would split off from the exact sequence $0\ra U\ra X\ra V\ra 0$, which is impossible since $U$ and $V$ are indecomposable.
\par
To show the sufficiency, it is enough to prove that $[Z,X]\leq [Z,U\oplus V]$ holds trivially for all other $Z$. If $Z\npreceq V$ then $[Z,V]=[Z,U]=0$. The case $\tau^{-1}U\npreceq Z$ can be treated using the Auslander-Reiten formula $[V,W]-[V,W']=[W,\tau V]-[W',\tau V]$ for $W$, $W'$ of the same dimension vector.
\end{proof}

\begin{conjecture}
In the above Proposition, the condition (3) is superfluous.
\end{conjecture}

\subsection{Comparison of cones: quiver-compatible case}
Taking the dimension vector induces a map $\mathbf{dim}:K_0^\oplus\ra K_0$. Let $\Lambda$ denote its kernel. Then we have the short exact sequence of free abelian groups $0\ra\Lambda\ra K_0^\oplus\ra K_0\ra 0$. Dualizing by taking $(-)^*:=\Hom(-,\mathbb{Z})$ we obtain an short exact sequence 
$$0\ra K_0^*\ra (K_0^\oplus)^*\ra\Lambda^*\ra 0.$$
Notice that $\Lambda^*$ consists of additive functions on $K_0^\oplus$ modulo those functions depending only on the dimension vectors of modules. 

\begin{lemma}\label{Lem:Lambda}
Every element in $\Lambda$ is of form $[W]-[V]$ where $W$ and $V$ are two $\mathbb{K}Q$-modules of the same dimension vector.
\end{lemma}

\begin{proof}
Let $x=\sum_{[U]}a_U [U]\in K_0^\oplus$ where the sum is over the isomorphism classes of indecomposable $\mathbb{K}Q$-modules. Gathering the positive and the negative coefficients among $a_U$, the element $x$ can be written as $[W]-[V]$ where $W$ and $V$ are $\mathbb{K}Q$-modules. That this element belongs to $\Lambda$ just means that $\mathbf{dim}V=\mathbf{dim}W$.
\end{proof}

\begin{theorem}\label{Thm:LSQuiver}
Let $Q$ be a quiver of type $\tt A$, $\tt D$ or $\tt E$ and $\underline{w}_0\in R(w_0)$ be compatible with $Q$. Then 
$$\mathcal{L}_{\underline{w}_0}^-(\g^L)=\mathcal{D}_{\underline{w}_0}^q(\g).$$
\end{theorem}

We prove this theorem in the rest of this section. Let $\Lambda_\mathbb{R}:=\Lambda\ts_\mathbb{Z}\mathbb{R}$.
\par
We define a cone $\mathcal{E}\subset\Lambda_\mathbb{R}$ generated by $[V]+[W]-[X]$ when there exists a short exact sequence $0\ra V\ra X\ra W\ra 0$. Let $\mathcal{D}:=\mathcal{E}^\vee\subset\Lambda_\mathbb{R}^*$ be the dual cone of $\mathcal{E}$, that is,
$$\mathcal{D}:=\{\vp\in\Lambda_\mathbb{R}^*\mid \vp(x)\geq 0\text{ for any }x\in \mathcal{E}\}.$$
That is to say, if $\vp\in\mathcal{D}$, then $\vp([V\oplus W])=\vp([V])+\vp([W])$ and for any short exact sequence $0\ra V\ra X\ra W\ra 0$, $\vp([X])\leq \vp([V])+\vp([W])$.

Using the terminology in \cite[Definition 2]{FFR16}, such a function $\vp$ is an admissible function on isomorphism classes of $\mathbb{K}Q$-modules. Then by \cite[Theorem 4(1)]{FFR16}, the cone $\mathcal{D}$ is freely generated by the functions $[U,-]$ for $U$ a non-projective indecomposable $\mathbb{K}Q$-module (since if $U$ is projective, then $[U,-]$ only depends on the dimension vector).

\begin{proposition}
We have $\mathcal{E}=\mathcal{D}^\vee$.
\end{proposition}

\begin{proof}
Recall that 
$$\mathcal{D}^\vee=\{x\in\Lambda\mid \vp(x)\geq 0\text{ for any }\vp\in\mathcal{D}\}.$$
First it is clear that $\mathcal{E}\subset (\mathcal{E}^\vee)^\vee=\mathcal{D}^\vee$. Conversely, suppose that $x\in\Lambda$ belongs to $\mathcal{D}^\vee$. By Lemma \ref{Lem:Lambda}, we can write $x=[W]-[V]$ for $\mathbb{K}Q$-modules $W$ and $V$ with $\mathbf{dim}W=\mathbf{dim}V$. Then $x\in\mathcal{D}^\vee$ means that for any $\mathbb{K}Q$-module $U$, $[U,V]\leq [U,W]$. Using the language in \cite{B}, this implies that $V$ degenerates to $W$. By \cite[Corollary 4.2]{B}, it is equivalent to $V\leq_{\mathrm{ext}}W$, which is defined as follows (see \cite[Section 1]{B}): there exists a sequence $V=V_0,V_1,\cdots,V_s=W$ of $\mathbb{K}Q$-modules such that for any $k=0,1,\cdots,s-1$, there exists a short exact sequence with middle term $V_k$, whose end terms add up to $V_{k+1}$. This shows that in the following expression
$$[W]-[V]=[V_s]-[V_{s-1}]+[V_{s-1}]-[V_{s-2}]+\cdots +[V_1]-[V_0],$$
for each $1\leq k\leq s$, $[V_k]-[V_{k-1}]\in\mathcal{E}$, hence $x\in\mathcal{E}$.
\end{proof}

To complete the proof of Theorem \ref{Thm:LSQuiver}, it suffices to apply the equivalence between (1) and (2) in Lemma \ref{Lem:degenerate}.

\subsection{Example}\label{Sec:example}
Let $\g=\mathfrak{sl}_{n+1}$ be of type $\tt A_n$ with positive roots 
$$\Delta_+=\{\alpha_{i,j}=\alpha_i+\cdots+\alpha_j\mid 1\leq i\leq j\leq n\}.$$
We compute the quantum degree cone $\mathcal{D}_{\underline{w}_0^\mathrm{max}}^q(\g)$ for
$$\underline{w}_0^{\mathrm{max}}:=(n,n-1,n,n-2,n-1,n,\cdots,1,\cdots,n).$$
This reduced decomposition is compatible with the following quiver of type $\tt A_n$:
\[\xymatrix{
1\ar[r] & 2 \ar[r] & \ar[r]\cdots & n-1 \ar[r]& n.}\]
For simplicity we denote $d_{i,j}:=d_{\alpha_{i,j}}$ for $1\leq i\leq j\leq n$. By Theorem \ref{Thm:LSQuiver}, it suffices to compute $\mathcal{L}_{\underline{w}_0^{\mathrm{max}}}^-(\g^L)$, which is given by:
$$\text{for }1\leq i\leq n-1,\ \ d_{i,i+1}\leq d_{i,i}+d_{i+1,i+1};$$
$$\text{for }1\leq i<j-1\leq n-1,\ \ d_{i,j}+d_{i+1,j-1}\leq d_{i,j-1}+d_{i+1,j}.$$

\section{Tropical flag varieties}\label{s5}

\subsection{Tropical varieties}
For a polynomial
$$f:=\sum_{\mathbf{a}\in\mathbb{N}^{n+1}}\lambda_\mathbf{a} x^\mathbf{a}\in\mathbb{C}[x_0,x_1,\cdots,x_n],$$
we define its support $\mathrm{supp}(f)$ as those tuples $\mathbf{a}\in\mathbb{N}^{n+1}$ such that $\lambda_\mathbf{a}\neq 0$. The tropicalization of $f$, denoted by $\mathrm{trop}(f)$, is a map from $\mathbb{R}^{n+1}$ to $\mathbb{R}$ defined by:
$$\mathrm{trop}(f)(\mathbf{w}):=\min\{\mathbf{w}\cdot\mathbf{a}\mid \mathbf{a}\in\mathrm{supp}(f)\}.$$
Let $\mathbf{1}:=(1,1,\cdots,1)\in\mathbb{R}^{n+1}$ and $H:=\mathbb{R}^{n+1}/\mathbb{R}\mathbf{1}$. 
\par
Let $I\subset\mathbb{C}[x_0,x_1,\cdots,x_n]$ be a radical homogeneous prime ideal and $V(I)$ be the (irreducible) projective variety associated to $I$. The tropical variety associated to $V(I)$ is defined by:
$$\mathrm{trop}(V(I))=\bigcap_{f\in I}\{\mathbf{w}\in H\mid \text{the minimum in }\mathrm{trop}(f)(\mathbf{w})\text{ is attained at least twice}\}.$$

The tropical variety $\mathrm{trop}(V(I))$ is a fan of dimension $\dim(V(I))$; all vectors $\mathbf{v}$ in the relative interior of a cone $C$ give the same initial ideal $\mathrm{in}_C(I):=\mathrm{in}_\mathbf{v}(I)$; vectors in different cones give different initial ideals.
\par
A cone $C$ in $\mathrm{trop}(V(I))$ is called a maximal prime cone if $C$ is a cone of maximal dimension, and $\mathrm{in}_C(I)$ is a prime binomial ideal. These maximal prime cones produce flat degenerations of the variety $V(I)$ to the toric variety $V(\mathrm{in}_C(I))$.

\subsection{Tropical complete flag varieties}
We fix an integer $n>1$. Let $\mathrm{SL}_n$ be the special linear group defined over $\mathbb{C}$ and $B\subset\mathrm{SL}_n$ be the Borel sub-group of upper triangular matrices in $\mathrm{SL}_n$. Let $\mathcal{F}l_n:=\mathrm{SL}_n/\mathrm{B}$ be the complete flag variety and $\mathbb{C}^n$ be the natural representation of $\mathrm{SL}_n$. We consider the Pl\"ucker embedding of $\mathcal{F}l_n$ into a product of projective spaces
$$\mathcal{F}l_n\lhook\joinrel\longrightarrow \prod_{k=1}^{n-1}\mathbb{P}(\Lambda^k\mathbb{C}^n).$$
We fix the standard basis $e_1,\cdots,e_n$ of $\mathbb{C}^n$ and $I=\{i_1,\cdots,i_k\}$ such that $1\leq i_1<\cdots<i_k\leq n$. The linear function 
$$P_I:=(e_{i_1}\wedge e_{i_2}\wedge\cdots\wedge e_{i_k})^*\in(\Lambda^k\mathbb{C}^n)^*$$
is called a Pl\"ucker coordinate on the product of projective spaces. The complete flag variety $\mathcal{F}l_n$ is defined by the Pl\"ucker relations in this product of projective spaces.
\par
Let $\mathrm{trop}(\mathcal{F}l_n)$ be the tropical variety associated to the complete flag variety with respect to the Pl\"ucker embedding. More details and some small dimension examples of $\mathrm{trop}(\mathcal{F}l_n)$ can be found in \cite{BLMM}.

\subsection{From quantum degree cone to tropical flag variety}
We denote $[n]:=\{1,2,\cdots,n\}$, $S:=2^{[n]}\backslash\{\emptyset,[n]\}$ and recall that
$$\underline{w}_0^{\mathrm{max}}:=(n,n-1,n,n-2,n-1,n,\cdots,1,\cdots,n).$$
\par
For $I=\{i_1,\cdots,i_k\}\in S$ with $i_1<\cdots<i_k$, we define two tuples $p_1<\cdots<p_\ell$ and $q_\ell<\cdots<q_1$ by: 
$$\{p_1,\cdots,p_\ell\}=[k]\backslash I,\ \ \{q_1,\cdots,q_\ell\}=I\backslash [k].$$
Let $\g=\mathfrak{sl}_n(\mathbb{C})$. We define a map 
$$\vp:\mathcal{D}_{\underline{w}_0^\mathrm{max}}^q(\g)\ra\mathbb{R}^S$$
by sending $\mathbf{d}\in \mathcal{D}_{\underline{w}_0^\mathrm{max}}^q(\g)$ to a vector $\vp(\mathbf{d})$ whose coordinates $\vp(\mathbf{d})_I$ is given by:
\begin{enumerate}
\item if $I=[k]$ for some $1\leq k\leq n-1$, $\vp(\mathbf{d})_I:=0$;
\item otherwise $\vp(\mathbf{d})_I:=d_{p_1,q_1-1}+\cdots+d_{p_\ell,q_\ell-1}$, where $d_{i,j}:=d_{\alpha_{i,j}}$.
\end{enumerate}
\bigskip

Then one of the main theorems in \cite{FFFM} can be rephrased as:

\begin{theorem}[\cite{FFFM}]
The map $\vp$ maps $\mathcal{D}_{\underline{w}_0^\mathrm{max}}^q(\g)$ bijectively onto the closure of a maximal cone in $\mathrm{trop}(\mathcal{F}l_n)$.
\end{theorem}

\begin{remark}
\begin{enumerate}
    \item The maximal prime cone obtained as the image of $\varphi$ gives the flat toric degeneration of $G/B$ to the toric variety associated to the Feigin-Fourier-Littelmann-Vinberg polytopes \cite{FFFM}.
    \item It is shown in \cite{M} that changing the directions of the inequalities in the example in Section \ref{Sec:example} identifies other maximal prime cones in tropical flag varieties.
\end{enumerate}
\end{remark}



\begin{thebibliography}{99}

\bibitem{ASS}
I.~Assem, A.~Skowronski, and D.~Simson,
\emph{Elements of the representation theory of associative algebras. Vol. 1. Techniques of representation theory.} London Mathematical Society Student Texts, 65. Cambridge University Press, Cambridge, 2006.

\bibitem{BFF}
T.~Backhaus, X.~Fang, G.~Fourier, 
\emph{Degree cones and monomial bases of Lie algebras and quantum groups}, 
Glasgow Mathematical Journal, Volume 59, Issue 3 September 2017, pp. 595--621.

\bibitem{B}
K.~Bongartz,
\emph{On degenerations and extensions of finite dimensional modules}, Adv. Math. 121, (1996), 245--287.

\bibitem{BLMM}
L.~Bossinger, S.~Lamboglia, K.~Mincheva, F.~Mohammadi,
\emph{Computing toric degenerations of flag varieties}, In: Smith G., Sturmfels B. (eds) Combinatorial Algebraic Geometry. Fields Institute Communications, pp.247--281, vol 80. Springer, New York, NY.

\bibitem{CMMG}
P.~Caldero, R.~Marsh, S.~Morier-Genoud,
\emph{Realisation of Lusztig cones}, Representation Theory 8, (2004), 458--478.

\bibitem{dCK}
C.~de Concini, V.G.~Kac, 
\emph{Representations of quantum groups at roots of 1}, in "Operator Algebras, Unitary Representations, Enveloping Algebras and Invariant Theory", Progress in Mathematics 92, Birkh\"auser, Boston-Basel-Berlin, (1990), 471--506.

\bibitem{FFFM}
X.~Fang, E.~Feigin, G.~Fourier, I.~Makhlin, { \it Weighted PBW degenerations and tropical flag varieties}, Commun. Contemp. Math., Volume 21, Number 1, 1850016 (2019).

\bibitem{FFR16}
X.~Fang, G.~Fourier, and M.~Reineke,{ \it PBW-Filtration on quantum groups of type $A_n$}, Journal of Algebra 449 (2016) 321-345.

\bibitem{Lus}
   G.~Lusztig,
   \emph{Introduction to quantum groups.}
   Reprint of the 1994 edition. Modern Birkh\"auser Classics. Birkh\"auser/Springer, New York, 2010.

\bibitem{Lus2}
G.~Lusztig,
\emph{Tight monomials in quantized enveloping algebras}, in "Quantum deformations of algebras and their representations" ed. A.Joseph et al., Isr. Math. Conf. Proc. 7 (1993), Amer. Math. Soc. 117--132.

\bibitem{LS1}
S.~Levendorskii,Y.~Soibelman, 
\emph{Some applications of quantum Weyl groups}, J. Geom. and
Physics, 7 (1990), 241--254.

\bibitem{M}
I.~Makhlin,
\emph{Gelfand--Tsetlin degenerations of representations and flag varieties},
preprint, arXiv:1809.02258.


\bibitem{Rin90}
 C. M.~Ringel,
   \emph{Hall algebras and quantum groups.}
   Invent. Math. \textbf{101}, (1990), 583-592.


\bibitem{Xi}
N.~ Xi,
\emph{A commutation formula for root vectors in quantized enveloping algebras}, Pacific Journal of Mathematics, Vol. 189 (1999), No. 1, 179--199.

\end{thebibliography}
\end{document}